\documentclass[runningheads]{llncs}

\usepackage{mathtools}
\usepackage[utf8]{inputenc}
\usepackage{amssymb}
\usepackage{amsfonts}
\usepackage{fancyhdr}
\usepackage{stmaryrd}
\usepackage{enumerate}
\usepackage{xspace}
\usepackage{tikz}
\usetikzlibrary{positioning}
\usetikzlibrary{decorations.pathreplacing}
\usetikzlibrary{arrows.meta}
\usetikzlibrary{calc}

\newcommand{\NC}{\mathsf{NC}}
\newcommand{\LogCFL}{\mathsf{LogCFL}}

\newcommand{\shlex}{\mathrm{shlex}}

\newcommand{\val}{\mathrm{val}}

\newcommand{\G}{\mathcal{G}}
\newcommand{\Z}{\mathbb{Z}}
\newcommand{\N}{\mathbb{N}}
\newcommand{\inv}{^{-1}}

\begin{document}

\title{Complexity of word problems for HNN-extensions} 

\author{Markus Lohrey\inst{1}}
\authorrunning{M. Lohrey}
\institute{Universit{\"a}t Siegen, Germany \\
\email{\{lohrey\}@eti.uni-siegen.de}}

\maketitle       

	\begin{abstract}
	The computational complexity of the word problem in HNN-extension of groups is studied. 
	HNN-extension is a fundamental construction in combinatorial group theory.
	It is shown that the word problem for an ascending HNN-extension of a group $H$ is logspace
	reducible to the so-called compressed word problem for $H$. The main result of the paper
	states that the word problem for an HNN-extension of a hyperbolic group $H$ with cyclic
	associated subgroups can be solved in polynomial time. This result can be easily
	extended to fundamental groups of graphs of groups with hyperbolic vertex groups 
	and cyclic edge groups.
	\end{abstract}
	
	\keywords{word problems \and HNN-extensions \and hyperbolic groups}
  \section{Introduction}
  
The study of computational problems in group theory goes back more than 100 years. 
In a seminal paper from 1911, Dehn posed three decision problems \cite{dehn11}: The {\em word problem},
the {\em conjugacy problem}, and the {\em isomorphism problem}. In this paper, we mainly deal with the word 
problem: It is defined for a finitely generated group $G$. This means that there
exists a finite subset $\Sigma \subseteq G$ such that every element of $G$ can be written as a finite product
of elements from $\Sigma$. This allows to represent elements of $G$ by finite words over the alphabet $\Sigma$.
For the word problem, the input consists of such a finite word $w \in \Sigma^*$ and the goal is to check whether $w$ represents the identity element of  $G$. 

In general the word problem is undecidable. By a classical result of Boone \cite{Boo59} and Novikov \cite{Nov58}, there exist finitely presented groups 
(finitely generated groups that can be defined by finitely many equations) with an undecidable word problem;
see \cite{Sti95} for an excellent exposition. On the positive side, there are many classes of groups with decidable
word problems. In his paper from 1912 \cite{Dehn12}, Dehn presented an algorithm that solves the word problem for 
fundamental groups of orientable closed 2-dimensional manifolds. This result was extended to one-relator 
groups (finitely generated groups that can be defined by a single equation)
by Dehn's student Magnus \cite{mag32}. Other important classes of groups with a decidable word problem are:
\begin{itemize}
\item automatic groups \cite{EpsCHLPT92} (including important classes like braid groups \cite{Artin25}, Coxeter groups \cite{bjofra05},
right-angled Artin groups \cite{Cha07}, hyperbolic groups \cite{Gro87}), 
\item finitely generated linear groups, i.e., finitely generated groups that can be faithfully represented by matrices over a field \cite{Rab60}
(including polycyclic groups and nilpotent groups), and
\item  finitely generated metabelian groups (they can be embedded in direct products
of linear groups \cite{Wehr80}).
\end{itemize}
With the rise of computational complexity theory in the 1960's, also the computational complexity of group theoretic
problems moved into the focus of research. From the very beginning, this field attracted researchers from mathematics
as well as computer science. It turned out that for many interesting classes of groups the word problem admits quite 
efficient algorithms. For instance, Lipton and Zalcstein \cite{LiZa77} (for fields of characteristic zero) and 
Simon \cite{Sim79} (for prime characteristic) proved in 1977 (resp., 1979) that deterministic logarithmic space (and hence polynomial time) suffices
to solve the word problem for a linear group. For automatic groups, 
the word problem can be solved in quadratic time \cite{EpsCHLPT92}, and for the subclass
of hyperbolic groups the word problem can be solved in linear time (even real time) \cite{Hol00} and belongs to 
the complexity class $\LogCFL$ \cite{Lo05ijfcs} (the closure of the context-free languages under logspace reductions). 
For one-relator groups, only a non-elementary algorithm is known for the word problem. For the so-called Baumslag group
(a particular one-relator group, whose word problem was believed to be computationally difficult) a polynomial time algorithm
for the word problem was found in  \cite{MyUsWo11}. Recently, the complexity has been further improved to $\mathsf{NC}$
 \cite{MaWe21}.

The complexity of the word problem is also preserved by several important group theoretic constructions, e.g.
graph products (which generalize free products and direct products) \cite{DiekertK16} and wreath products \cite{Waa90}.
Two other important constructions in group theory are 
{\em HNN-extensions} and {\em amalgamated free products}. A theorem of Seifert and van Kampen links these constructions to
algebraic topology. Moreover, HNN-extensions are used in all modern proofs for the undecidability of the word problem
in finitely presented groups. For a base group $H$ with two isomorphic subgroups 
$A$ and $B$ and an isomorphism $\varphi \colon A \to B$, the corresponding
HNN-extension is the group
\begin{equation} \label{HNN}
G=\langle H,t \mid  t^{-1} a t = \varphi(a) \, (a \in A) \rangle.
\end{equation}
Intuitively, it is obtained by adjoing to $H$ a new generator $t$ (the {\em stable letter}) 
in such a way that conjugation of $A$ by $t$ realizes 
$\varphi$. The subgroups $A$ and $B$ are also 
called the {\em associated subgroups}.
If $H$ has a decidable word problem, $A$ and $B$ are finitely generated subgroups of $H$,
and the subgroup membership problems for $A$ and $B$ are decidable, then also the word problem for 
$G$ in \eqref{HNN} is decidable via {\em Britton reduction} \cite{Britt63} (iterated application of rewriting steps 
$t^{-1} a t \to \varphi(a)$ and $t b t^{-1} = \varphi^{-1}(b)$ for $a \in A$ and $b \in B$).
For the special case where $A=B$ and $\varphi$ is the identity, it is shown in  \cite{Waa90} that the word
problem for the HNN-extension $G$ in \eqref{HNN} is $\NC^1$-reducible to the following problems:
(i) the word problem for $H$,  (ii) the word problem for the free group of rank two, and (iii) the subgroup membership
problem for $A$. 
On the other hand, it is not clear whether this result can be extended to arbitrary HNN-extensions (even if we 
allow polynomial time Turing reductions instead of $\NC^1$-reductions).
A concrete open problem is the complexity of the word problem for an HNN-extension 
$\langle F,t \mid  t^{-1} a t = \varphi(a) \, (a \in A) \rangle$ of a free group $F$
with finitely generated associated subgroups $A$ and $B$. The word problem for a free group is known to be in logspace
(it is a linear group) \cite{LiZa77} and the subgroup membership problem for finitely generated subgroups of a free group can be solved in polynomial
time \cite{AvMa84a}. The problem with Britton reduction in the group $\langle F,t \mid  t^{-1} a t = \varphi(a) \, (a \in A) \rangle$ is that every
Britton reduction step may increase the length of the word by a constant multiplicative factor. This may lead to words of exponential length.
One might try to solve this problem by representing the exponentially long words by so-called 
straight-line programs (context-free grammars that produce a single word). This idea works for the word problems of automorphism
groups and certain group extensions \cite[Section~4.2]{Loh14}. But it is not clear whether the words that arise from Britton reduction can be compressed
down to polynomial size using straight-line programs. The problem arises from the fact that both $A$ and $B$ might be proper subgroups
of $H$. On the other hand, if one of the associated subgroups $A$ and $B$ coincides with the base group $H$ ($G$ is then called
an {\em ascending HNN-extension}) then one can show that the word problem for $G$ is logspace-reducible to the so-called {\em compressed 
word problem} for $H$ (Theorem~\ref{thm-ascending}). The latter problem asks whether a given straight-line program that produces a 
word over the generators of $H$ evaluates to the group identity of $H$. The compressed word problem is known to be solvable
in polynomial time for nilpotent groups, virtually special groups, and hyperbolic groups. For every linear
group one still has a randomized polynomial time algorithm for the compressed word problem; see \cite{Loh14} for details.

Our main result deals with HNN-extensions, where the associated subgroups $A$ and $B$ are allowed to be proper subgroups of the base
group $H$ but are cyclic (i.e., generated by a single element) and {\em undistored} in $H$ (the latter is defined in Section~\ref{sec-HNN-cyclic}).
We show that in this situation the word problem for $G$ is polynomial time Turing-reducible to the {\em compressed power problem} for $H$
(Theorem~\ref{thm-general-WP-HNN}). In the compressed power problem for $H$, the input consists of two elements $p,q \in H$, where
$p$ is given explicitly as a word over a generating set and $q$ is given in compressed form by a straight-line program over a generating
set. The question is whether there exists an integer $z \in \Z$ such that $p^z = q$ in $H$. Moreover, in the positive case we also want to 
compute such a $z$. 

Our main application of Theorem~\ref{thm-general-WP-HNN} concerns hyperbolic groups.
We show that the compressed power problem for a hyperbolic group
can be solved in polynomial time (Theorem~\ref{thm-CPP-hyp}). For this, we make use of the well-known fact that
cyclic subgroups of hyperbolic groups are undistorted. As a consequence of 
Theorems~\ref{thm-CPP-hyp} and~\ref{thm-general-WP-HNN}, the word problem for an HNN-extension of 
a hyperbolic group with cyclic associated subgroups can be solved in polynomial time (Corollary~\ref{coro-hyp}).
One should remark that HNN-extensions of hyperbolic groups with cyclic associated subgroups are in general
not even automatic; a well-known example is the Baumslag-Solitar group $\mathsf{BS}(1,2) = \langle a,t \mid t^{-1} a t = a^2 \rangle$
\cite[Section~7.4]{EpsCHLPT92}.

Corollary~\ref{coro-hyp} can be generalized to fundamental groups of graphs of groups (which generalize HNN-extensions
and amalgamated free products) with hyperbolic vertex groups and
cyclic edge groups (Corollary~\ref{coro-hyp2}). For the special case where all vertex groups are free,
the existence of a polynomial time algorithm for the word problem has been stated in \cite[Remark~5.11]{Weiss15b} without proof.
For a fundamental group of a graph of groups, where all vertex groups are copies of $\mathbb{Z}$, the word problem can be even solved in logspace \cite{Weiss16}.

\section{Groups} \label{sec-groups}
 
 For real numbers $a \leq b$ we denote with $[a,b] = \{ r \in \mathbb{R} \mid a \leq r \leq b\}$
 the closed interval from $a$ to $b$. For $k, \ell \in \N$ we write $[k:\ell]$ for $\{ i \in \N \mid k \leq i \leq \ell\}$.
We use standard notations for words (over some alphabet $\Sigma$). As usual, the empty word is denoted with $\varepsilon$.
 Given a word $w = a_1 a_2 \cdots a_n$ (where $a_1, a_2, \ldots, a_n \in \Sigma$)
and numbers $i,j \in \mathbb{N}$ with $1 \leq i \leq j$ we define $w[i:j] = a_i a_{i+1} \cdots a_{\min\{j,n\}}$.

 For a group $G$ and a subset $\Sigma \subseteq G$, we denote with $\langle \Sigma \rangle$ the subgroup
of $G$ generated by $\Sigma$. It is the smallest subgroup of $G$ containing $\Sigma$.
If $G = \langle \Sigma \rangle$ then $\Sigma$ is a {\em generating set} for $G$. 
The group $G$ is {\em finitely generated (f.g.)} if it has a finite generating set. 
We mostly consider f.g.~groups in this paper. 

Assume that $G = \langle \Sigma \rangle$ and
let $\Sigma^{-1} = \{ a^{-1} \mid a \in \Sigma \}$. For a word
$w = a_1 \cdots a_n$ with $a_i \in \Sigma \cup \Sigma^{-1}$ we 
define the word $w^{-1} = a_n^{-1} \cdots a_1^{-1}$. This defines an involution
on the free monoid $ (\Sigma \cup \Sigma^{-1})^*$. We obtain a surjective monoid homomorphism 
$\pi \colon (\Sigma \cup \Sigma^{-1})^* \to G$
that preserves the involution: $\pi(w^{-1}) = \pi(w)^{-1}$. 
We also say that the word $w$ represents the group element
$\pi(w)$. For words $u,v \in (\Sigma \cup \Sigma^{-1})^*$    
we say that $u=v$ in $G$ if $\pi(u) = \pi(v)$. For $g \in G$ one defines 
$|g|_\Sigma = \min \{ |w| : w \in \pi^{-1}(g)\}$ as the length of a shortest word over 
$\Sigma \cup \Sigma^{-1}$ representing $g$. If $\Sigma$ is clear, we 
also write $|g|$ for $|g|_\Sigma$.
 If $\Sigma = \Sigma^{-1}$ then $\Sigma$ is a finite {\em symmetric} generating set for $G$.

We will describe groups by presentations. In general, if $H$ is a group and $R \subseteq H$ 
is a set of so-called {\em relators}, then we denote with $\langle H \mid R \rangle$ the quotient
group $H/N_R$, where $N_R$ is the smallest normal subgroup of $H$ with $R \subseteq N_R$.
Formally, we have $N_R = \langle \{ h r h^{-1} \mid h \in H, r \in R \} \rangle$.
For group elements $g_i, h_i \in H$ ($i \in I$) we also write
$\langle H \mid g_i = h_i \; (i \in I) \rangle$ for the group 
$\langle H \mid \{ g_i h_i^{-1} \mid i \in I \} \rangle$.

In most cases, one takes a free group for the group $H$ from the previous paragraph.
Fix a set $\Sigma$ and let $\Sigma^{-1}  = \{ a^{-1} \mid a \in \Sigma\}$
be a set of formal inverses of the elements in $\Sigma$ with $\Sigma \cap \Sigma^{-1} = \emptyset$.
A word $w \in (\Sigma \cup \Sigma^{-1})^*$  is called \emph{freely reduced} if it
neither contains a factor $a a^{-1}$ nor $a^{-1} a$ for $a \in \Sigma$. For every word
$w \in (\Sigma \cup \Sigma^{-1})^*$ there is a unique freely reduced
word that is obtained from $w$ by deleting factors $a a^{-1}$ and $a^{-1} a$ ($a \in \Sigma$) as long as 
possible. The {\em free group} generated by $\Sigma$ consists of all freely reduced words together with the multiplication 
defined by $u \cdot v = \mathrm{nf}(uv)$ for $u, v$ freely reduced.  
For a set $R \subseteq F(\Sigma)$ of relators we also write 
$\langle \Sigma \mid R \rangle$ for the group $\langle F(\Sigma) \mid R \rangle$. 
Every group $G$ that is generated by $\Sigma$ can be written
as $\langle \Sigma \mid R \rangle$ for some $R \subseteq F(\Sigma)$. A group
$\langle \Sigma \mid R \rangle$ with $\Sigma$ and $R$ finite is called {\em finitely presented}, and 
the pair $(\Sigma, R)$ is a {\em presentation} for the group $\langle \Sigma \mid R \rangle$.
Given two groups $G_1 = \langle \Sigma_1 \mid R_1 \rangle$ and $G_2 = \langle \Sigma_2 \mid R_2 \rangle$, where w.l.o.g.~$\Sigma_1 \cap \Sigma_2 = \emptyset$,
we define their {\em free product} $G_1 * G_2 = \langle \Sigma_1 \cup \Sigma_2 \mid R_1 \cup R_2 \rangle$.


Consider a group $G$ with the finite symmetric generating set $\Sigma$.
The \emph{word problem for $G$ w.r.t. $\Sigma$} is the following
decision problem:
\begin{description}
\item[input:] a word $w\in \Sigma^*$.
\item[question:]  does $w=1$ hold in $G$?
\end{description}
It is well known that if $\Sigma'$ is another finite
symmetric generating set for $G$, then the word problem for $G$ w.r.t. $\Sigma'$
is logspace many-one reducible to the word problem for $G$ w.r.t. $\Sigma$.
This justifies one to speak just of the word problem for the group $G$.

\paragraph{\bf HNN-extensions.}

HNN-extension is an extremely important operation for constructing groups that arises in all parts of combinatorial group theory. 
Take a group $H$ and a generator $t \not\in H$, from which we obtain the free product $H * \langle t \rangle \cong
H * \Z$. Assume that $A, B \leq H$ are two isomorphic subgroups of $H$ and let $\varphi \colon A \to B$ be an isomorphism.
Then, the group
$\langle H * \langle t \rangle \mid t^{-1} a t = \varphi(a) \; (a \in A) \rangle$
is called the {\em HNN-extension of A with associated subgroups $A$ and $B$} (usually, the isomorphism $\varphi$ is not mentioned explicitly).
The above HNN-extension is usually written as $\langle H, t \mid t^{-1} a t = \varphi(a) \; (a \in A) \rangle$.
Britton \cite{Britt63} proved the following fundamental result for HNN-extensions.
Let us fix a finite symmetric generating set $\Sigma$ for $H$.

\begin{theorem}[Britton's lemma \cite{Britt63}] \label{thm-britton}
Let $G = \langle H, t \mid t^{-1} a t = \varphi(a) \; (a \in A) \rangle$ be an HNN-extension.
If a word $w \in (\Sigma \cup \{t,t^{-1}\})^*$ represents the identity of $G$ then $w$ contains a factor of the form 
$t^{-1} u t$ (resp., $t u t^{-1}$), where $u \in \Sigma^*$ represents an element of $A$ (resp., $B$).
\end{theorem}
A subword of the form $t^{-1} u t$ (resp., $t u t^{-1}$), where $u \in \Sigma^*$ represents an element of $A$ (resp., $B$)
is also called a {\em pin}.

A simple corollary of Britton's lemma is that $H$ is a subgroup of the HNN-extension $\langle H, t \mid t^{-1} a t = \varphi(a) \; (a \in A) \rangle$.
Britton's lemma can be also used to solve the word problem for an HNN-extension $\langle H, t \mid t^{-1} a t = \varphi(a) \; (a \in A) \rangle$.
For this we need several assumptions:
\begin{itemize}
\item The word problem for $H$ is decidable.
\item  There is an algorithm that decides whether a given word 
 $u \in \Sigma^*$ represents an element of $A$ (resp., $B$). 
 \item Given a word $u \in \Sigma^*$ that represents an element $a \in A$ (resp., $b \in B$), one can
 compute a word $v \in \Sigma^*$ that represents the element $\varphi(a)$ (resp., $\varphi^{-1}(b)$).
 Let us denote this word $v$ with $\varphi(u)$ (resp., $\varphi^{-1}(u)$).
\end{itemize}
 Then, given a  word $w \in (\Sigma \cup \{t,t^{-1}\})^*$ one replaces pins $t^{-1} u t$ (resp., $t u t^{-1}$) by
 $\varphi(u)$ (resp., $\varphi^{-1}(u)$) in any order, until no more pins occur.  If the final word does not belong to $\Sigma^*$
 then we have $w \neq 1$ in the HNN-extension. If the final word belongs to $\Sigma^*$ then one uses the algorithm
 for the word problem of $H$ to check whether it represents the group identity.
 This algorithm is known as {\em Britton reduction}.
 
 An HNN-extension $G = \langle H, t \mid t^{-1} a t = \varphi(a) \; (a \in A) \rangle$ with 
$\varphi \colon A \to B$ is called {\em ascending} if $A = H$ (it is also called the mapping torus of $\varphi$). 
Note that we do not require $B = H$. Ascending HNN-extensions play an important role in many
group theoretical results. For instance, Bieri and Strebel \cite{BiStre78} proved that if $N$ is a normal
subgroup of a finitely presented group $G$ such that $G/N \cong \Z$ then $G$ is an ascending HNN-extension
of a finitely generated group or contains a free subgroup of rank two.

\paragraph{\bf Hyperbolic groups.}

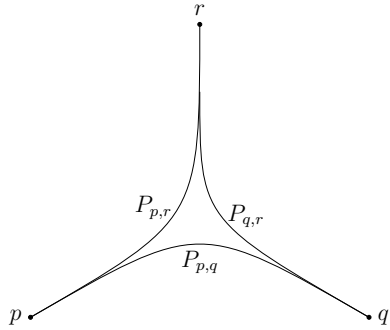
\begin{figure}[t]
 \centering{
 \scalebox{.75}{
\begin{tikzpicture}
\tikzstyle{small} = [circle,draw=black,fill=black,inner sep=.25mm]
\node (p) at (0,0) [small, label=left: \large $p$] {};
\node (q) at (6,0) [small, label=right: \large  $q$] {};
\node (r) at (3,5.19615) [small, label=above: \large  $r$] {};
\draw (0,0) .. controls (3,1.73205)  .. node[pos=.5,below=0mm] {\large $P_{p,q}$} (6,0);
\draw (0,0) .. controls (3,1.73205)  .. node[pos=.5,above=0mm, left=0mm] {\large  $P_{p,r}$} (3,5.19615);
\draw (3,4) .. controls (3,1.73205)  .. node[pos=.5,above=0mm, right=0mm] {\large $P_{q,r}$} (6,0);
\end{tikzpicture}
}}
\caption{\label{fig-geo-tri}The shape of a geodesic triangle in a hyperbolic group}
\end{figure}

Let $G$ be a f.g.~group with the finite symmetric generating set $\Sigma$.
The  {\em Cayley-graph} of $G$ (with respect to $\Sigma$) is the undirected  graph $\Gamma = \Gamma(G)$ with node set
$G$ and all edges $(g,ga)$ for $g \in G$ and $a \in \Sigma$. We view $\Gamma$ as a geodesic metric space,
where every edge $(g,ga)$ is identified with a unit-length interval. It is convenient to label the directed edge
from $g$ to $ga$ with the generator $a$.
The distance between two points $p,q$ is denoted with $d_\Gamma(p,q)$.
Note that $|g|_\Sigma = d_\Gamma(1,g)$ for $g \in G$. For $r \geq 0$,  let
$\mathcal{B}_r(1) = \{ g \in G \mid d_\Gamma(1,g) \leq r \}$.

Paths can be defined in a very general way for metric spaces, but we only need paths
that are induced by words over $\Sigma$.
Given a word $w \in \Sigma^*$ of length $n$, one obtains a unique path $P[w] \colon [0,n] \to \Gamma$,
which is a continuous mapping from the real interval $[0,n]$ to $\Gamma$.
It maps the subinterval $[i,i+1] \subseteq [0,n]$  isometrically 
onto the edge $(g_i, g_{i+1})$ of $\Gamma$, where $g_i$ (resp., $g_{i+1})$ is the group element
represented by the word $w[1:i]$ (resp., $w[1:i+1]$).
The path $P[w]$ starts in $1 = g_0$ and ends in $g_n$ (the group element represented by $w$).
We also say that $P[w]$ is the unique path that starts in $1$ and is labelled with the word $w$.
More generally, for $g \in G$ we denote
with $g \cdot P[w]$ the path that starts in $g$  and is labelled with $w$.
When writing $u \cdot P[w]$ for a word $u \in \Sigma^*$, we mean the path
$g \cdot P[w]$, where $g$ is the group element represented by $u$.

Let $\lambda,\zeta>0$, $\epsilon \geq 0$ be real constants. 
A path  $P \\colon [0,n] \to \Gamma$ of the above form is {\em geodesic} if $d_\Gamma(P(0),P(n)) = n$; it is
a $(\lambda,\epsilon)$-{\em quasigeodesic} if for all points
$p = P(a)$ and $q = P(b)$ we have $|a-b| \leq \lambda \cdot d_\Gamma(p,q) + \epsilon$;
and it is $\zeta$-{\em local} $(\lambda,\epsilon)$-{\em quasigeodesic} if for all points
$p = P(a)$ and $q = P(b)$ with $|a-b| \le \zeta$ we have $|a-b| \leq \lambda \cdot d_\Gamma(p,q) + \epsilon$.

A word $w \in \Sigma^*$ is geodesic if the path $P[w]$ is geodesic, which means
that there is no shorter word representing the same group element from $G$.
Similarly, we define the notion of ($\zeta$-local) $(\lambda,\epsilon)$-quasigeodesic words.
A word $w \in \Sigma^*$ is {\em shortlex reduced} if it is the length-lexicographically smallest word 
that represents the same group element as $w$. For this, we have to fix an arbitrary linear
order on $\Sigma$. 
Note that if $u = xy$ is shortlex reduced then $x$ and $y$ are shortlex reduced  too.
For a word $u \in \Sigma^*$ we denote
with $\shlex(u)$ the unique shortlex reduced word that represents the same group element as $u$
(the underlying group $G$ will be always clear from the context).

A {\em geodesic triangle} consists of three points $p,q,r \in G$ and geodesic paths $P_1 = P_{p,q}$, $P_2 = P_{p,r}$, $P_3 = P_{q,r}$
(the three sides of the triangle),
where $P_{x,y}$ is a geodesic path from $x$ to $y$. We call a geodesic triangle {\em $\delta$-slim}
for $\delta \geq 0$, if for all $i \in \{1,2,3\}$, every point on $P_i$ has distance at most 
$\delta$ from a point on $P_j \cup P_k$, where $\{j,k\} = \{1,2,3\} \setminus \{i\}$.
The group $G$ is called  {\em $\delta$-hyperbolic}, if  every geodesic triangle is 
$\delta$-slim. Finally, $G$ is  hyperbolic, if it is  $\delta$-hyperbolic
for some $\delta \geq 0$. Figure~\ref{fig-geo-tri} shows the shape of a geodesic triangle in a hyperbolic group.
Finitely generated free groups are for instance $0$-hyperbolic.
The property of being hyperbolic is independent of the chosen generating set $\Sigma$. 
The word problem for every hyperbolic group can be decided in real time \cite{Hol00}. 

Fix a $\delta$-hyperbolic group $G$ with the finite symmetric generating set $\Sigma$
for the rest of the section, and let $\Gamma$ be the corresponding geodesic metric space.
Let us write $|g|$ for $|g|_\Sigma$. We need the following lemma:

\begin{lemma}[c.f.~\mbox{\cite[8.21]{ghys1990groupes}}] \label{lemma-cyclic-words-quasi-geo}
Let $g \in G$ be of infinite order and let $n \geq 0$. 
Let $u$ be a geodesic word representing $g$.
Then the word  $u^n$ is $(\lambda,\epsilon)$-quasigeodesic, where 
$\lambda = N|g|$, $\epsilon = 2N^2 |g|^2  + 2N |g|$ and $N = |\mathcal{B}_{2\delta}(1)|$. 
\end{lemma}
Consider two paths $P_1 \colon [0,n_1] \to \Gamma$, $P_2 \colon [0,n_2] \to \Gamma$ and let $\kappa \in \mathbb{R}$,
$\kappa \geq 0$. 
The paths $P_1$ and $P_2$ {\em asynchronously $\kappa$-fellow travel} if there exist two continuous non-decreasing mappings
$\varphi_1 \colon [0,1] \to [0,n_1]$ and $\varphi_2 \colon [0,1] \to [0,n_2]$ such that $\varphi_1(0) = \varphi_2(0) = 0$, $\varphi_1(1) = n_1$,
$\varphi_2(1) = n_2$ and for all $0 \leq t \leq 1$, $d_\Gamma(P_1(\varphi_1(t)), P_2(\varphi_2(t))) \leq \kappa$.
Intuitively, this means that one can travel along the paths $P_1$ and $P_2$ asynchronously with variable speeds such that
at any time instant the current points have distance at most $\kappa$. If $P_1$ and $P_2$ 
 asynchronously $\kappa$-fellow travel, then by slightly increasing $\kappa$ one obtains
  a subset $E \subseteq [0:n_1] \times [0:n_2]$
with (i) $(0,0), (n_1, n_2) \in E$, $d_\Gamma(P_1(i), P_2(j)) \leq \kappa$ for all $(i,j) \in E$ and (iii)
if $(i,j) \in E \setminus \{ (n_1, n_2) \}$ then 
$(i+1,j) \in E$ or $(i,j+1) \in E$. We write $P_1 \approx_\kappa P_2$ in this case. 
Intuitively, this means that  a ladder graph as shown in Figure~\ref{fig-ladder} exists, where the edges connecting the horizontal $P_1$- and $P_2$-labelled 
paths represent paths of length $\le \kappa$ that connect elements from $G$. 

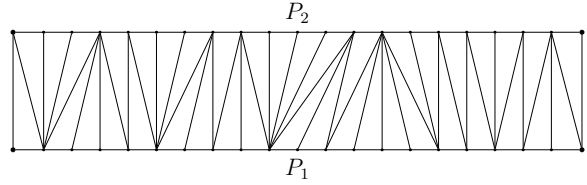
\begin{figure}[t]
 \centering{
 \scalebox{.75}{
\begin{tikzpicture}
  \tikzstyle{small} = [circle,draw=black,fill=black,inner sep=.25mm]
  \tikzstyle{tiny} = [circle,draw=black,fill=black,inner sep=.1mm]
  \tikzstyle{zero} = [circle,inner sep=0mm]

    \node[small] (a1) {} ;
    \node[small,  right = 10cm of a1] (a2) {};
    \node[small,  above = 2cm of a1] (b1) {};
    \node[small,  right = 10cm of b1] (b2) {};
    
    \node[zero,  below = 2.5mm of a1] (a1') {} ;
    \node[zero,  below = 2.5mm of a2] (a2') {} ;
    \node[zero,  above = 2.5mm of b1] (b1') {} ;
    \node[zero,  above = 2.5mm of b2] (b2') {} ;
    
       \draw [-] (a1) to   
      node[pos=.5,below=.5mm]  {\large $P_1$}   
      node[pos=.05, tiny] (1) {}  
      node[pos=.1, tiny] (2) {}  
      node[pos=.15, tiny] (3) {} 
      node[pos=.2, tiny] (4) {}  
      node[pos=.25, tiny] (5) {}  
      node[pos=.3, tiny] (6) {}  
      node[pos=.35, tiny] (7) {}  
      node[pos=.4, tiny] (8) {}  
      node[pos=.45, tiny] (9) {}  
      node[pos=.5, tiny] (10) {}  
      node[pos=.55, tiny] (11) {}  
      node[pos=.6, tiny] (12) {}  
      node[pos=.65, tiny] (13) {}  
      node[pos=.7, tiny] (14) {}  
      node[pos=.75, tiny] (15) {}  
      node[pos=.8, tiny] (16) {}  
      node[pos=.85, tiny] (17) {} 
      node[pos=.9, tiny] (18) {} 
      node[pos=.95, tiny] (19) {}  
        (a2);
        
        \draw [-] (b1) to   
      node[pos=.5,above=.5mm]  {\large $P_2$}   
      node[pos=.05, tiny] (1') {}  
      node[pos=.1, tiny] (2') {}  
      node[pos=.15, tiny] (3') {} 
      node[pos=.2, tiny] (4') {}  
      node[pos=.25, tiny] (5') {}  
      node[pos=.3, tiny] (6') {}  
      node[pos=.35, tiny] (7') {}  
      node[pos=.4, tiny] (8') {}  
      node[pos=.45, tiny] (9') {}  
      node[pos=.5, tiny] (10') {}  
      node[pos=.55, tiny] (11') {}  
      node[pos=.6, tiny] (12') {}  
      node[pos=.65, tiny] (13') {}  
      node[pos=.7, tiny] (14') {}  
      node[pos=.75, tiny] (15') {}  
      node[pos=.8, tiny] (16') {}  
      node[pos=.85, tiny] (17') {} 
      node[pos=.9, tiny] (18') {} 
      node[pos=.95, tiny] (19') {}  
        (b2);

       \draw [-] (a1) to  (b1);
      \draw [-] (1) to  (b1);               
      \draw [-] (1) to (1');
      \draw [-] (1) to  (2');
      \draw [-] (1) to  (3');
      \draw [-] (2) to  (3');
      \draw [-] (3) to  (3');
      \draw [-] (4) to  (3');
      \draw [-] (4) to  (4');
      \draw [-] (5) to  (4');
      \draw [-] (5) to  (5');    
      \draw [-] (5) to  (6');
      \draw [-] (5) to  (7');
      \draw [-] (6) to  (7');
      \draw [-] (7) to  (7');
      \draw [-] (7) to  (8');
      \draw [-] (8) to  (8');
      \draw [-] (9) to  (8');
      \draw [-] (9) to  (9');
      \draw [-] (9) to  (10');
      \draw [-] (9) to  (11');
      \draw [-] (9) to  (12');
       \draw [-] (10) to  (12');
       \draw [-] (11) to  (12');
       \draw [-] (11) to  (13');
       \draw [-] (12) to  (13');
       \draw [-] (13) to  (13');
       \draw [-] (14) to  (13');
       \draw [-] (15) to  (13');
       \draw [-] (15) to  (14');
       \draw [-] (15) to  (15');
       \draw [-] (16) to  (15');
       \draw [-] (16) to  (16');
       \draw [-] (17) to  (16');
       \draw [-] (17) to  (17');
       \draw [-] (17) to  (18');
       \draw [-] (18) to  (18');
       \draw [-] (18) to  (19');
       \draw [-] (19) to  (19');
        \draw [-] (a2) to  (19');
    \draw [-] (a2) to  (b2);
   \end{tikzpicture}}}
\caption{\label{fig-ladder} Paths that asynchronously $\kappa$-fellow travel}
  \end{figure}

\begin{lemma}[c.f.~\mbox{\cite[Lemma 1]{MyNi14}}] \label{lemma-asynch-fellow-travel}
Let $P_1$ and $P_2$ be $(\lambda,\epsilon)$-quasigeodesic paths in $\Gamma$ and assume that
$P_i$ starts in $g_i$, ends in $h_i$, and $d_\Gamma(g_1,g_2), d_\Gamma(h_1,h_2) \leq h$.
Then there is a constant
$\kappa = \kappa(\delta, \lambda,\epsilon, h) \geq h$ such that $P_1 \approx_\kappa P_2$.
\end{lemma}
For the following lemmas we fix two further constants: 
\begin{equation} \label{constants-L-K}
L = 34\delta+2 \quad\text{ and } \quad K = |\mathcal{B}_{4\delta}(1)|^2 .
\end{equation}

\begin{lemma}[c.f.~\mbox{\cite[Lemma~3.1]{EpsteinH06}}] \label{lemma-Ep-Ho1}
Let $u = u_1 u_2$ be shortlex reduced, where $|u_1| \le |u_2| \le |u_1|+1$.
Let $\tilde u = \shlex(u_2 u_1)$.
If $|\tilde u| \ge 2L+1$ then for every $n \geq 0$, the word $\tilde{u}^n$ is $L$-local $(1,2\delta)$-quasigeodesic.
\end{lemma}
The following lemma is not stated explicitly in \cite{EpsteinH06} but is shown in \cite[Section~3.2]{EpsteinH06}
(where the main argument is attributed to Delzant).

\begin{lemma}[c.f.~\cite{EpsteinH06}] \label{lemma-Ep-Ho2}
Let $u$ be geodesic such that $|u| \ge 2L+1$ and for every $n \geq 0$, the word $u^n$ is $L$-local $(1,2\delta)$-quasigeodesic.
Then, in time $\mathcal{O}(|u|)$ one can compute $c \in \mathcal{B}_{4\delta}(1)$ and an integer $1 \leq m \leq K$ such that $(\shlex(c^{-1} u^m c))^n$
is geodesic for all $n \geq 0$.
\end{lemma}

\subsection{Compressed words and the compressed word problem}

Straight-line programs offer succinct representations of long words that contain many repeated substrings.  
We here review the basics, referring to~\cite{Loh14} for a more in-depth introduction. 

Fix a finite alphabet $\Sigma$.  A \emph{straight-line program} $\mathcal{G}$ (SLP for short) 
is a context-free grammar that generates exactly one word $\val(\mathcal{G}) \in \Sigma^*$.  
More formally, an SLP over $\Sigma$ is a triple $\mathcal{G} = (V, S, \rho)$ where
\begin{itemize}
\item
$V$ is a finite set of \emph{variables}, disjoint from $\Sigma$, 
\item
$S \in V$ is the \emph{start variable}, and 
\item
$\rho \colon V \to  (V \cup \Sigma)^*$ is the {\em right-hand side mapping}, which is acyclic in the sense that the binary relation 
$\{ (A,B) \in V \times V \mid \mbox{$B$ appears in $\rho(A)$} \}$ is acyclic.
\end{itemize}
We define the \emph{size} $|\mathcal{G}|$ of $\mathcal{G}$ as $\sum_{A \in V} |\rho(A)|$.
The \emph{evaluation} function $\val = \val_\mathcal{G} \colon (V \cup \Sigma)^* \to \Sigma^*$ is the unique homomorphism
between free monoids such that (i) $\val(a) = a$ for $a \in \Sigma$, and
(ii) $\val(A) = \val(\rho(A))$ for $A \in V$.
We finally take $\val(\mathcal{G}) = \val(S)$.  
We call $\val(\mathcal{G})$ the word defined by the SLP $\mathcal{G}$.

\begin{example}
\label{exam-double}
Let $\Sigma = \{a, b\}$ and fix $n \geq 0$.  
We define $\mathcal{G}_n = (\{A_0, \ldots, A_n\}, A_n, \rho)$,
where $\rho(A_0) = ab$ and $\rho(A_{i+1}) = A_i A_i$ for $0 \leq i \leq n - 1$.
It is an SLP of size $2(n+1)$. We have $\val(A_0) = ab$ and more generally $\val(A_i) = (ab)^{2^i}$.  
Thus $\val(\mathcal{G}_n) = \val(A_n) = (ab)^{2^n}$.
\end{example}
The SLP $\mathcal{G} = (V, S, \rho)$ is \emph{trivial} if $S$ is the only variable and $\rho(S) = \varepsilon = \val(\mathcal{G})$.  
An SLP is in \emph{Chomsky normal form} if it is either trivial or all right-hand sides $\rho(A)$ are of the form $a \in \Sigma$ or $BC$ with $B, C \in V$.
There is a linear-time algorithm that transforms a given SLP $\mathcal{G}$ into an SLP $\mathcal{G}'$ in Chomsky normal such that
$\val(\G) = \val(\G')$; see~\cite[Proposition~3.8]{Loh14}.

The following theorem is the technical main result from \cite{HoltLS19}:
\begin{theorem}[c.f.~\cite{HoltLS19}] \label{main-HLS}
Let $G$ be a hyperbolic group with the finite symmetric generating set $\Sigma$.
Given an SLP $\G$ over $\Sigma$ one can compute in polynomial time an 
SLP $\mathcal{H}$ over $\Sigma$ such that $\val(\mathcal{H}) = \shlex(\val(\G))$.
\end{theorem}
If $G$ is a f.g.~group with the finite and symmetric generating set $\Sigma$, then we define the {\em compressed word problem}
of $G$ as the following problem:
\begin{description}
\item[input:] an  SLP $\mathcal{G}$  over $\Sigma$.
\item[question:]  does $\val(\mathcal{G})$ represent the group identity of $G$?
\end{description}
An immediate consequence of Theorem~\ref{main-HLS} is the following result:

\begin{theorem}[c.f.~\cite{HoltLS19}] \label{CWP-hyp}
The compressed word problem for a hyperbolic group can be solved in polynomial time.
\end{theorem}
The compressed word problem turns out to be useful for the solution of the word problem 
for an ascending HNN-extension:

\begin{theorem} \label{thm-ascending}
Let $H$ be a finitely generated group.
The word problem for an ascending HNN-extension 
$G=\langle H, t \mid t^{-1} a t = \varphi(a) \; (a \in H) \rangle$ is logspace-reducible to the compressed
word problem for $H$.
\end{theorem}

\begin{proof}
The proof is similar to corresponding results for automorphism groups and semi-direct products \cite[Section~4.2]{Loh14}.
Let us fix a finite and (w.l.o.g.) symmetric generating set $\Sigma$ for $H$. 
Consider an input word $w \in (\Sigma \cup \{t,t^{-1}\})^*$ and  write 
$w = w_0 t^{\epsilon_1} w_1 t^{\epsilon_2} w_2 \cdots t^{\epsilon_n} w_n$,
where $w_i \in \Sigma^*$ for $0 \leq i \leq n$ and $\epsilon_i \in \{-1,1\}$ for $1 \leq i \leq n$.
Let $s_k = \sum_{i=1}^k \epsilon_i$ for $0 \leq k \leq n$.
Clearly, $w = 1$ in $G$ if and only if $t^{-n} w t^{n} = 1$ in $G$ if and only if 
\begin{equation} \label{eq-ascending}
t^{-n} \bigg(\prod_{i=0}^n t^{s_i} w_i t^{-s_i} \bigg) t^{s_n+n} = \bigg(\prod_{i=0}^n  t^{s_i-n} w_i t^{n-s_i} \bigg) t^{s_n} = 
\bigg(\prod_{i=0}^n \varphi^{n-s_i}(w_i) \bigg) t^{s_n} = 1
\end{equation}
in $G$. Here, we identify $\varphi$ with a homomorphism $\tilde\varphi \colon \Sigma^* \to \Sigma^*$ such that for every
$a \in \Sigma$, the word $\tilde\varphi(a)$ represents the group element $\varphi(a) \in G$.
By Britton's lemma, \eqref{eq-ascending} is equivalent to $s_n = 0$ (this can be checked in logspace) and
$\prod_{i=0}^n \varphi^{n-s_i}(w_i) = 1$ in $H$. The latter is an instance of the compressed word problem.
We can easily (in logspace) compute an SLP for the word $\prod_{i=0}^n \varphi^{n-s_i}(w_i)$, see e.g.~\cite[Lemma 3.12]{Loh14}.
\qed
\end{proof}
We will also need a generalization of straight-line programs, known as composition systems
\cite[Definition~8.1.2]{Hag00}
(in \cite{Loh14} they are called \emph{cut straight-line programs}). 
A \emph{composition system} over $\Sigma$ is a tuple $\mathcal{G} = (V, S, \rho)$, with $V$ and $S$ as for an SLP, and where we also allow, as right-hand sides for $\rho$, expressions of the form $B[i:j]$, with $B \in V$ and $i,j \in \mathbb{N}$, $1 \leq i \leq j$.  
The numbers $i$ and $j$ are stored in binary encoding.
We again require $\rho$ to be acyclic.
When $\rho(A) = B[i:j]$ we define $\val(A) = \val(B)[i:j]$.
We define the \emph{size} $|\mathcal{G}|$ of the composition system $\mathcal{G}$ as 
the total number of occurrences of symbols from $V \cup \Sigma$ in all right-hand sides. Hence, a right-hand $B[i:j]$
contributes $1$ to the size, and we ignore the numbers $i,j$. Adding the bit lengths of the numbers $i$ and $j$ to the size
$|\mathcal{G}|$ would only lead to a polynomial blow-up for $|\mathcal{G}|$.
To see this, first normalize the composition system so that all right-hand sides have the form
$a$, $BC$ or $B[i:j]$ with $a \in \Sigma$ and $B,C \in V$; analogously to the Chomsky normal form of SLPs
this can be achieved in polynomial time. If $n$ is the number of variables of the resulting composition system, then 
every variable produces a string of length at most $2^n$. Hence, we can assume that all numbers $i,j$ that appear in
a right-hand side $B[i:j]$ are of bit length $\mathcal{O}(n)$. 

We can now state an important result of Hagenah; see~\cite[Algorithmus~8.1.4]{Hag00} as well as~\cite[Theorem~3.14]{Loh14}.
\begin{theorem}
\label{thm-hagenah}
There is a polynomial-time algorithm that, given a composition system $\mathcal{G}$, computes an SLP $\mathcal{G}'$ such that 
$\val(\mathcal{G}) = \val(\mathcal{G}')$.
\end{theorem}
It will be convenient to allow in composition systems also more complex right-hand sides.
For instance $(ABC)[i:j] D$ would first concatenate the strings produced from $A$, $B$, and $C$. From the resulting 
string the substring from position $i$ to position $j$ is cut out and this substring is concatenated with the string 
produced by $D$.

\section{The compressed power problem}

In the next section we want to study the word problem in HNN-extensions with cyclic associated subgroups.
For this, the following computational problem turns out to be important.
Let $G$ be a f.g.~group with the finite symmetric generating set $\Sigma$.
We define the {\em compressed power problem} for $G$
as the following problem:
\begin{description}
\item[input:] a word $w\in \Sigma^*$ and an SLP $\G$ over $\Sigma$.
\item[output:]  the binary coding of an integer $z \in \Z$ such that $w^z = \val(\G)$ in $G$
if such an integer exists, and {\it no} otherwise.
\end{description}

\begin{theorem} \label{thm-CPP-hyp}
For every hyperbolic group $G$, the compressed power problem can be solved in polynomial time.
\end{theorem}

\begin{proof}
Fix the word $w \in \Sigma^*$ and the SLP $\G = (V,\rho,S)$ over $\Sigma$, w.l.o.g. in Chomsky normal form.
We have to check whether the equation 
\begin{equation} \label{eq-w^x}
w^z = \val(\G)
\end{equation}
has a solution in $G$, and compute in the positive case a solution $z \in \Z$. Let $g$ be the group element represented by $w$.

In a hyperbolic group $G$ the order of torsion elements is bounded by a fixed constant 
that only depends on $G$, see also the proof of \cite[Theorem~6.7]{MyNiUs14}.
This allows to check in polynomial time whether $g$ has finite order in $G$.
If $g$ has finite order, say $d$, then it remains to check for all $0 \leq i \leq d-1$ whether
$w^i = \val(\G)$ in $G$, which can be done in polynomial time by Theorem~\ref{CWP-hyp}.
This solves the case where $g$ has finite order in $G$.

Now assume that $g$ has infinite order in $G$. Then \eqref{eq-w^x} has at most one solution.
By considering also the equation $(w^{-1})^z = \val(\G)$,  it suffices to search for a solution $z \in \N$.
We can also assume that $w$ is shortlex-reduced. Using techniques from \cite{EpsteinH06} one can
further ensure that for every $n \in \N$, $w^n$ is $(\lambda,\epsilon)$-quasigeodesic for fixed constants $\lambda$ and $\epsilon$
that only depend on the group $G$: 

\paragraph{Reduction to the case with $w^n$ $(\lambda,\epsilon)$-quasigeodesic for all $n$:} 
Let us fix the two constants $L$ and $K$ from \eqref{constants-L-K} and define further constants:
\begin{equation} \label{def-constants}
N = |\mathcal{B}_{2\delta}(1)|, \quad  \lambda = N (2L+1), \quad \text{and} \quad \epsilon = 2N^2 (2L+1)^2  + 2N (2L+1).
\end{equation}
We factorize $w$ uniquely as $w = u v$ where $|u| \le |v| \le |u|+1$, and let
$\tilde{w} = \shlex(vu)$. Note that $|\tilde{w}| \le |w|$.
Let $\tilde g$ be the group element represented by
$\tilde{w}$. Since $\tilde g$ is conjugated
to $g$, also $\tilde g$ has infinite order.  By Lemma~\ref{lemma-cyclic-words-quasi-geo},
for every $n \geq 0$, the word $\tilde{w}^n$ is $(\lambda',\epsilon')$-quasigeodesic
for  $\lambda' = N |\tilde{w}|$, $\epsilon' = 2N^2 |\tilde{w}|^2  + 2N |\tilde{w}|$.
If $|\tilde{w}| < 2L+1$ then $\tilde{w}^n$ is 
$(\lambda, \epsilon)$-quasigeodesic for the constants $\lambda$ and $\epsilon$ from \eqref{def-constants}.
We then replace the equation $w^z = \val(\G)$ in \eqref{eq-w^x} by the equivalent equation $u \tilde{w}^{z} u^{-1} = \val(\G)$
(or $\tilde{w}^{z}  = u^{-1} \val(\G) u$).
To see the equivalence of these two equations, 
note that for every $n \geq 0$, $u \tilde{w}^n u^{-1} = u (v u)^n u^{-1} = 
(u v)^n = w^n$ in $G$.

Now assume that $|\tilde{w}| \ge 2L+1$.
By Lemma~\ref{lemma-Ep-Ho1}, $\tilde{w}^n$ is $L$-local $(1,2\delta)$-quasi\-geo\-desic for every $n \geq 0$.
By Lemma~\ref{lemma-Ep-Ho2}, one can compute in polynomial time
$c \in \mathcal{B}_{4\delta}(1)$ and an integer $1 \leq m \leq K$ such that $(\shlex(c^{-1} \tilde{w}^{m} c))^n$
is geodesic (and hence $(1,0)$-quasigeodesic) for all $n \geq 0$. 
We then produce for every number $0 \leq d \leq m-1$ a new equation 
$u \tilde{w}^{d} c (\shlex(c^{-1} \tilde{w}^{m} c))^{z} c^{-1} u^{-1} = \val(\G)$,
or, equivalently, $(\shlex(c^{-1} \tilde{w}^{m} c))^{z} = c^{-1} \tilde{w}^{-d} u^{-1}\val(\G)uc$.
If we denote this equation with $\mathcal{E}_d$ then we see that 
\begin{itemize}
\item if $w^n = \val(\G)$ in $G$ for $n \in \N$ then $\lfloor n/m \rfloor$ is a solution of $\mathcal{E}_d$, where $d = n \bmod m$, and
\item if $n$ is a solution of $\mathcal{E}_d$ for some $0 \leq d \leq m-1$, then $w^{n \cdot m + d} = \val(\G)$ in $G$.
\end{itemize}
Hence, it suffices to check for each of the constantly many equations $\mathcal{E}_d$ ($0 \leq d \leq m-1$) whether it has a solution and to compute the solution if it exists.

The above consideration shows that we can restrict to the case of an equation $w^z = \val(\G)$,
where $w$ represents a group element of infinite order and
for every $n \in \N$, $w^n$ is $(\lambda,\epsilon)$-quasigeodesic for fixed constants $\lambda$ and $\epsilon$.

\medskip
\noindent
Finally, by Theorem~\ref{main-HLS} we can also assume that the word $\val(\G)$ (and hence every word
$\val(X)$ for $X$ a variable of $\G$) is shortlex-reduced. Hence, if $w^z = \val(\G)$ for some $z \in \N$, then
by Lemma~\ref{lemma-asynch-fellow-travel} we have  $P[w^z] \approx_\kappa P[\val(\G)]$ for a fixed constant $\kappa$ that only depends
on $G$. We proceed in two steps.

\medskip
\noindent
{\em Step 1.}
We compute in polynomial time for all variables $X \in V$ of the SLP $\G$, 
all group elements $a,b \in \mathcal{B}_\kappa(1)$ (there are only constantly many), and
all factors $w'$ of $w$ a bit $\beta[X,a,b,w'] \in \{0,1\}$ which is set to $1$ if and only if (i) $\val(X) = a w' b$ in $G$
and (ii) $P[\val(X)] \approx_\kappa a \cdot P[w']$. 

We compute these bits $\beta[X,a,b,w']$ in a bottom-up process where we begin with variables $X$ such that $\rho(X)$ is a terminal
symbol and end with the start variable $S$. So, let us start with a variable $X$ such that $\rho(X) = c \in \Sigma$ and let
$a,b,w'$ as above. Then we have to check whether $c = aw'b$ in $G$ and 
$P[c] \approx_\kappa a \cdot P[w']$. The former can be checked in linear time (it is an instance of the word problem) and
the latter can be done in polynomial time as well: we have to check whether the 
path $a \cdot P[w']$ splits into two parts, where all vertices in the first (resp., second) part belong to  $\mathcal{B}_\kappa(1)$
(resp., $\mathcal{B}_\kappa(c)$).

Let us now consider a variable $X$ with $\rho(X) = YZ$ such that all bits $\beta[Y,a,b,w']$ and $\beta[Z,a,b,w']$ have been computed.
Let us fix $a,b \in \mathcal{B}_\kappa(1)$ and a factor $w'$ of $w$. We have $\beta[X,a,b,w']=1$ if and only if there exists a factorization
$w' = w'_1 w'_2$ and $c \in \mathcal{B}_\kappa(1)$ such that $\beta[Y,a,c,w'_1]=1$ and $\beta[Z,c^{-1},b,w'_2]=1$. This allows us to compute
$\beta[X,a,b,w']$ in polynomial time.

\medskip
\noindent
{\em Step 2.}
We compute in polynomial time for all variables $X \in V$,  all group elements $a,b \in \mathcal{B}_\kappa(1)$, 
all proper suffixes $w_2$ of $w$, and all proper prefixes $w_1$ of $w$  the unique number $z = z[X,a,b,w_2,w_1] \in \N$ (if it exists) such that (i) $\val(X) = a w_2 w^z w_1 b$ in $G$
and (ii) $P[\val(X)] \approx_\kappa a \cdot P[w_2 w^z w_1]$.
If such an integer $z$ does not exist we set $z[X,a,b,w_2,w_1] = \infty$. Note that the integers $z[X,a,b,w_2,w_1]$ are unique since
$w$ represents a group element of infinite order. We represent $z[X,a,b,w_2,w_1]$ in binary encoding.
As in step 1, the computation of the numbers $z[X,a,b,w_2,w_1]$ begins with variables $X$ 
such that $\rho(X)$ is a terminal symbol and ends with the start variable $S$.
The bits $\beta[X,a,b,w']$ from step 1 are needed in the computation. 
 
Let us start with a variable $X$ such that $\rho(X) = c \in \Sigma$ and let
$a,b,w_2, w_1$ as above. We have to consider the equation $c = a w_2 w^z w_1 b$, or, equivalently, $w^z = u$ 
where $u = \shlex(w_2^{-1} a^{-1} c b^{-1} w_1^{-1})$.
We can compute the word $u$ in linear time. Since $w^n$ is  $(\lambda,\epsilon)$-quasigeodesic for all $n \in \N$, every $n \in \N$ with $w^n = u$ in $G$
has to satisfy $n \cdot |w| \leq \lambda \cdot |u| + \epsilon$, i.e., $n \leq |w|^{-1} ( \lambda \cdot |u| + \epsilon )$.
Hence, we can check for all $0 \le n \leq |w|^{-1} ( \lambda \cdot |u| + \epsilon )$  whether $w^n = u$ in $G$. If we do not find a solution, we set $z[X,a,b,w_1,w_2] = \infty$. If we find a (unique) solution
$n$, we can check in polynomial time whether $P[\val(X)]  = P[c] \approx_\kappa a \cdot P[w_2 w^n w_1]$ as above for 
$P[\val(X)] \approx_\kappa a \cdot P[w']$ in step 1.

\begin{figure}[t]
  \centering{
    \scalebox{1}{
      \begin{tikzpicture}
        \tikzstyle{small} = [circle,draw=black,fill=black,inner sep=.15mm]
        \tikzstyle{zero} = [circle,inner sep=0mm]

        \node[small] (1) {} ;
        \node[small,  right = 4cm of 1] (a) {} ;
        \node[small,  above = 0.9cm of a] (b) {} ;
        \node[small,  right = 8cm of 1] (2) {};
        \node[small,  above = 2cm of 1] (3) {};
        \node[small,  right = 8cm of 3] (4) {};
    
        \node[zero,  below = .6mm of 1] (1') {};
        \node[zero,  below = .6mm of 2] (2') {};
        \node[zero,  above = .6mm of 3] (3') {};
        \node[zero,  above = .6mm of 4] (4') {};
      
        \draw [->] (3) to [out=-27, in=-180] node[pos = 0.5, below = -0.7mm] {$u$} (b);
        \draw [->] (b) to [out=0, in=-153] node[pos = 0.5, below = -0.7mm]  {$v$} (4);
        \draw [->] (3') to [out=-27, in=-153] node[pos = 0.5, above = -0.7mm]{$w_2 w^z w_1$} (4'); 
      
        \draw [->] (1) edge node[left=-.7mm]{$a$} (3);
        \draw [<-] (2) edge node[right=-.7mm]{$b$} (4);
        \draw [->] (1') edge node[pos = 0.5, below=-.7mm]{$\val(X)$} (2');
        \draw [->] (1) edge node[above=-.7mm]{$\val(Y)$} (a);
        \draw [->] (a) edge node[above=-.7mm]{$\val(Z)$} (2);
        \draw [<-] (a) edge node[left=-.7mm]{$c$} (b);

  \end{tikzpicture}}}
  \caption{Situation in the proof of Lemma~\ref{thm-CPP-hyp}.}
  \label{fig-quad}
\end{figure}
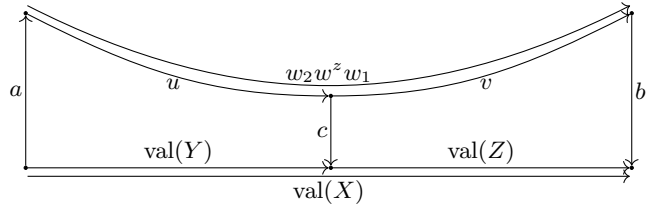

Now, let $X$ be a variable with $\rho(X) = YZ$ such that all values $z[Y,a,b,w_2,w_1]$ and $z[Z,a,b,w_2,w_1]$ have been computed.
Let us fix $a,b \in \mathcal{B}_\kappa(1)$, a proper suffix $w_2$ of $w$ and a proper prefix $w_1$ of $w$.
Note that if $\val(X) = a w_2 w^z w_1 b$ in $G$
and $P[\val(X)] \approx_\kappa a \cdot P[w_2 w^z w_1]$ for some $z \in \N$, then there 
must exist $c \in \mathcal{B}_\kappa(1)$ and a factorization  $w_2 w^z w_1 = uv$  such that
\begin{itemize}
\item $\val(Y) = a u c$ and $\val(Z) = c^{-1} v b$ in $G$,
\item $P[\val(Y)] \approx_\kappa a \cdot P[u]$, and 
\item $P[\val(Z)] \approx_\kappa c^{-1} \cdot P[v]$; see Figure~\ref{fig-quad}.
\end{itemize}
 For the factorization $w_2 w^z w_1 = u v$, one of the following cases has to hold:
\begin{itemize}
\item There is a factorization $w_2 = u w'_2$ such that $v = w'_2 w^{z} w_1$. We then have 
$\beta[Y, a,c,u]=1$ and $z = z[Z,c^{-1},b,w'_2,w_1]$. Vice versa, if $\beta[Y, a,c,u]=1$, $z[Z,c^{-1},b,w'_2,w_1]<\infty$ and
$w_2 = u w'_2$
then $$z[X,a,b,w_2,w_1] = z[Z,c^{-1},b,w'_2,w_1].$$
\item There is a factorization $w_1 = w'_1 v$ such that $u = w_2 w^z w'_1$. We then have
$z = z[Y,a,c,w_2,w'_1]$ and $\beta[Z, c^{-1},b,v]=1$.
Vice versa, if $z[Y,a,c,w_2,w'_1]<\infty$, $\beta[Z, c^{-1},b,v]=1$ and $w_1 = w'_1 v$ then
$$z[X,a,b,w_2,w_1] = z[Y,a,c,w_2,w'_1].
$$
\item There are $z_1, z_2 \in \N$ such that 
$u = w_2 w^{z_1}$, $v = w^{z_2} w_1$, and $z = z_1+z_2$. We then have 
$z = z[Y,a,c,w_2,\varepsilon]+z[Z,c^{-1},b,\varepsilon,w_1]$. Vice versa, if 
$z[Y,a,c,w_2,\varepsilon] <  \infty$, $z[Z,c^{-1},b,\varepsilon,w_1] < \infty$ then
$$z[X,a,b,w_2,w_1] = z[Y,a,c,w_2,\varepsilon]+z[Z,c^{-1},b,\varepsilon,w_1].
$$
\item There are $z_1, z_2 \in \N$ and a proper factorization $w = w'_1 w'_2$ such that  $u = w_2 w^{z_1} w'_1$,
$v = w'_2 w^{z_2} w_1$, and $z = z_1+z_2 + 1$. 
We then have $z = z[Y,a,c,w_2,w'_1]+z[Z,c^{-1},b,w'_2,w_1]+1$.
Vice versa, if 
$z[Y,a,c,w_2,w'_1] <  \infty$, $z[Z,c^{-1},b,w'_2,w_1] < \infty$ and $w = w'_1 w'_2$ then
$$z[X,a,b,w_2,w_1] = z[Y,a,c,w_2,w'_1]+z[Z,c^{-1},b,w'_2,w_1]+1.
$$
\end{itemize}
From these observations it is straightforward to compute in polynomial time all values $z[X,a,b,w_2,w_1]$ from the 
values $z[Y,a,c,w_2,w'_1]$ and $z[Z,c^{-1},b,w'_2,w_1]$ where $c \in \mathcal{B}_\kappa(1)$, $w'_1$ is a proper prefix
of $w$ and $w'_2$ is a proper suffix of $w_2$.

Finally, $z[S,1,1,\varepsilon,\varepsilon]$ is the unique solution 
of equation \eqref{eq-w^x} if $z[S,1,1,\varepsilon,\varepsilon] < \infty$.
This completes the proof of the theorem.
\qed
 \end{proof}

\section{HNN-extensions with cyclic associated subgroups}
\label{sec-HNN-cyclic}

Let $H$ be a f.g.~group and fix a generating set $\Sigma$ for $H$.
We say that a cyclic subgroup $\langle g \rangle \leq H$ is {\em undistorted} in $H$ if there exists a constant $\delta$
such that for every $h \in \langle g \rangle$ there exists $z \in \Z$ with $h = g^z$ and $|z| \leq \delta \cdot |h|_\Sigma$
(this definition does not depend on the choice of $\Sigma$).\footnote{The concept of undistorted subgroups is defined for arbitrary finitely generated subgroups but we will need it only for  the cyclic case.} 
This is clearly the case if $\langle g \rangle$ is finite. 

Note that if $g, h \in H$ are elements of the same order then the group 
$\langle H, t \mid t^{-1} g t = h \rangle$ is the HNN-extension $\langle H, t \mid t^{-1} a t = \varphi(a)\; (a \in \langle g \rangle) \rangle$,
where $\varphi \colon \langle g \rangle \to  \langle h \rangle$ is the isomorphism with $\varphi(g^z) = h^z$ for all $z \in \Z$.
In the following theorem we consider a slight extension of the word problem for such an HNN-extension
$G=\langle H, t \mid t^{-1} g t = h \rangle$ which we call the {\em semi-compressed word problem} for $G$.
In this problem the input is a sequence 
$\mathcal{G}_0 t^{\epsilon_1} \mathcal{G}_1 t^{\epsilon_2} \mathcal{G}_2 \cdots t^{\epsilon_n} \mathcal{G}_n$ 
where every $\mathcal{G}_i$ ($0 \leq i \leq n$) is an SLP (or a composition system) over the alphabet $\Sigma$
and $\epsilon_i \in \{-1,1\}$ for $1 \leq i \leq n$. The question is whether 
$\val(\mathcal{G}_0) t^{\epsilon_1} \val(\mathcal{G}_1) t^{\epsilon_2} \val(\mathcal{G}_2) \cdots t^{\epsilon_n} \val(\mathcal{G}_n)=1$ in 
$G$. 

\begin{theorem} \label{thm-general-WP-HNN}
Let $H$ be a fixed f.g.~group and let $g,h\in H$ be elements with the same order in $H$ (so that
the cyclic subgroups $\langle g \rangle$ and $\langle h \rangle$ are isomorphic)
such that $\langle g \rangle$ and $\langle h \rangle$ are 
undistorted. Then the semi-compressed word problem for the HNN-extension $\langle H, t \mid t^{-1} g t = h \rangle$
is polynomial-time Turing-reducible to the compressed power problem for $H$.
\end{theorem}

\begin{proof}
The case where $\langle g \rangle$ and $\langle h \rangle$ are both finite is easy. 
In this case, by the main result of \cite{HauLo11}, even the 
 compressed word problem for $\langle H, t \mid t^{-1} g t = h \rangle$ is polynomial 
 time Turing-reducible to the compressed word problem for $H$, which is a special case of the compressed power problem. 

Let us now assume that $\langle g \rangle$ and $\langle h \rangle$ are infinite.
Fix a symmetric finite generating set $\Sigma$ for $H$.
Let $W = \mathcal{G}_0 t^{\epsilon_1} \mathcal{G}_1 t^{\epsilon_2} \mathcal{G}_2 \cdots t^{\epsilon_n} \mathcal{G}_n$  be
an input for the semi-compressed word problem for $\langle H, t \mid t^{-1} g t = h \rangle$, 
where $\G_i$ is a composition system over $\Sigma$ for $0 \leq i \leq n$ and $\epsilon_i \in \{-1,1\}$ for $1 \leq i \leq n$.
Basically, we do Britton reduction in any order on the word 
$\val(\mathcal{G}_0) t^{\epsilon_1} \val(\mathcal{G}_1) t^{\epsilon_2} \val(\mathcal{G}_2) \cdots t^{\epsilon_n} \val(\mathcal{G}_n)$. 
The number of Britton reduction steps is bounded by $n/2$.
After the $i$-th step we have a sequence $U = \mathcal{H}_0 t^{\zeta_1} \mathcal{H}_1 t^{\zeta_2} \mathcal{H}_2 \cdots t^{\zeta_m} \mathcal{H}_m$
where $m \leq n$, $\mathcal{H}_i = (V_i,S_i,\rho_i)$ is a composition system over $\Sigma$, and  $\zeta_i \in \{-1,1\}$.
Let $u_i = \val(\mathcal{H}_i)$,  $s_i = |\mathcal{H}_i|$  and define $s(U) = m + \sum_{i=0}^m s_i$,
which is a measure for the encoding length of $U$.
We then search for an $1 \leq i \leq m-1$ such that one of the following two cases holds:
\begin{enumerate}[(i)]
\item $\zeta_i = -1$, $\zeta_{i+1} = 1$ and there is an $\ell \in \Z$ such that $u_i = g^\ell$ in $H$.
\item $\zeta_i = 1$, $\zeta_{i+1} = -1$ and there is an $\ell \in \Z$ such that $u_i = h^\ell$ in $H$.
\end{enumerate}
Using oracle access to the compressed power problem for $H$ we can check in polynomial time whether
one of these cases holds and compute the corresponding integer $\ell$.
We then replace the subsequence $\mathcal{H}_{i-1} t^{\zeta_i} \mathcal{H}_i t^{\zeta_{i+1}} \mathcal{H}_{i+1}$ by 
a composition system $\mathcal{H}'_i$ where $\val(\mathcal{H}'_i)$ is 
$u_{i-1} h^\ell u_{i+1}$ in case (i) and $u_{i-1} g^\ell u_{i+1}$ in case (ii).
Let $U'$ be the resulting sequence.
It remains to bound $s(U')$. For this we have to bound the size of the composition system $\mathcal{H}'_i$.
Assume that $\zeta_i = -1$, $\zeta_{i+1} = 1$, and $u_i = g^\ell$ in $H$
(the case where $\zeta_i = 1$, $\zeta_{i+1} = -1$ and $u_i = h^\ell$ in $H$ is analogous). 
It suffices to show that $h^\ell$ can be produced by a composition system $\mathcal{H}''_i$ of size
$s_i + O(1)$. Then we can easily bound the size of $\mathcal{H}'_i$ by $s_{i-1} + s_i + s_{i+1} + O(1)$,
which yields $s(U') \leq s(U)+O(1)$.  This shows that 
 every sequence $V$ that occurs during the Britton reduction satisfies
$S(V) \leq S(W) + O(n)$ (recall that $W$ is the initial sequence and that the number of Britton reductions
is bounded by $n/2$).

Fix the constant $\delta$ such that for every $g' \in \langle g \rangle$ the unique (since $g$ has infinite order)
$z \in \mathbb{Z}$ with $g' = g^z$ satisfies $|z| \leq \delta \cdot |g'|_\Sigma$.
Hence, we have $|\ell| \leq \delta \cdot |u_i|$. W.l.o.g. we can assume that $\delta \in \mathbb{N}$.
The variables of $\mathcal{H}''_i$ are the variables of $\mathcal{H}_i$ plus two new variables $A_h$ and $S'_i$.
Define a morphism $\eta$ by $\eta(a) = A_h$ for all $a \in \Sigma$ and $\eta(A)=A$ for every variable $A$ of $\mathcal{H}_i$.
We define the right-hand side mapping $\rho''_i$ of $\mathcal{H}''_i$ by:
$\rho''_i(A_h) = h$ if $\ell \geq 0$ and $\rho''_i(A_h) = h^{-1}$ if $\ell < 0$, 
$\rho''_i(S'_i) = (S_i^\delta)[1: |\ell| \cdot |h|]$ and $\rho''_i(A) = \eta(\rho_i(A))$ 
for all variables $A$ of  $\mathcal{H}_i$. 
Note that $S_i^\delta$ derives to $h^{\delta \cdot |u_i|}$ if $\ell \geq 0$ and to 
 $h^{-\delta \cdot |u_i|}$ if $\ell < 0$. Since $|\ell| \leq \delta \cdot |u_i|$, 
$(S_i^\delta)[1: |\ell| \cdot |h|]$ derives to $h^{\ell}$.
The start variable of $\mathcal{H}''_i$ is $S'_i$.
The size of $\mathcal{H}''_i$ is $s_i  + |h| + \delta = s_i + O(1)$, since $|h|$ and $\delta$ are constants.
\qed
\end{proof}
A subgroup of a hyperbolic group is 
undistorted if and only if it is quasiconvex \cite[Lemma~1.6]{Min04}. That cyclic subgroups in hyperbolic groups are quasiconvex was shown
by Gromov \cite[Corollary~8.1.D]{Gro87}. Hence, infinite cyclic subgroups of a hyperbolic group are undistorted. 
Together with Theorems~\ref{thm-CPP-hyp} and \ref{thm-general-WP-HNN} we get:

\begin{corollary} \label{coro-hyp}
Let $H$ be a hyperbolic group and let $g,h \in H$ have the same order.
Then the word problem for $\langle H, t \mid t^{-1} g t = h \rangle$ can be solved in polynomial time.
\end{corollary}

\section{Generalization to graph of groups}\label{appendix-graph-groups}

We can slightly generalize Corollary~\ref{coro-hyp}. For this we need the definition of a graph of groups and its
fundamental group; a detailed introduction can be found in~\cite{Serre03}.

By a graph $\Gamma$, we mean a graph in the sense of Serre~\cite{Serre03}.  So
$\Gamma$ consists of a set $V$ of vertices, $E$ of edges, a function
$\alpha\colon E\to V$ selecting the initial vertex of an edge and a
fixed-point-free involution on $E$ written $e\mapsto e\inv$ (thus, $(e^{-1})^{-1} = e$ and $e \neq e^{-1}$ for all edges $e$).  This involution
extends to paths in the natural way.  One defines the terminal vertex function
$\omega\colon E\to V$ by $\omega(e)=\alpha(e\inv)$. 

A \textit{graph of groups} $(G,Y)$  consists of a graph $Y = (V,E)$ and
\begin{itemize}
\item[(i)] for each vertex $v \in V$, a group $G_v$,
\item[(ii)] for each edge $e \in E$, a group $G_e$ such that $G_e = G_{e\inv}$,
\item[(ii)] for each edge $e \in E$, monomorphisms $\alpha_e\colon G_e \to G_{\alpha(e)}$ and
$\omega_e \colon G_e \to G_{\omega(e)}$ such that $\alpha_e = \omega_{e\inv}$ for
all $e \in E$.
\end{itemize}
We assume that the groups $G_v$ intersect only in the identity, and
that they are disjoint from the edge set $E$. For each $v \in
V$, let $\langle \Sigma_v \mid R_v \rangle$ be a presentation
for $G_v$, with the different generating sets $\Sigma_v$ disjoint.  Let $\Delta$ be a 
set containing exactly one edge from each orbit of the involution $e\mapsto e\inv$ on $E$; we identify $E$ and $\Delta \cup \Delta^{-1}$ when convenient. Let $\Sigma$ be the (disjoint) union of all
the sets $\Sigma_v$ and $\Delta$.   We
define a group $F(G,Y)$ by the presentation
\begin{equation*}
F(G,Y)  =  \langle \Sigma \mid  R_v \; (v \in V), \;
e\omega_e(g)e\inv = \alpha_e(g) \; (e \in E, g \in G_e)  \rangle.
\end{equation*}
Fix a vertex $v_0 \in V(Y)$. A word in $w \in (\Sigma \cup \Sigma\inv)^*$ is of
\textit{cycle type at $v_0$} if it is of the form $w = w_0 e_1 w_1
e_2 w_2 \dots e_n w_n$ where:
\begin{itemize}
\item[(i)]  $e_i \in E$ for all $1\leq i\leq n$,
\item[(ii)] $e_1\cdots e_n$ is a path in $Y$ starting and ending at $v_0$,
\item[(iii)] $w_0\in (\Sigma_{v_0} \cup \Sigma_{v_0}\inv)^*$, and
\item[(iv)] for $1\leq i\leq n$, $w_i\in (\Sigma_{\omega(e_i)} \cup \Sigma_{\omega(e_i)}\inv)^*$.
\end{itemize}
The images in $F(G,Y)$ of the words of cycle type at $v_0$ form a
subgroup $\pi_1(G, Y, v_0)$ of $F(G,Y)$, called the
\textit{fundamental group of $(G,Y)$ at $v_0$}. The fundamental group of a connected graph of groups is (up to isomorphism)
independent of the choice of vertex $v_0$; hence we simply write $\pi_1(G, Y)$.
 It can be also defined via a spanning tree of $Y$, but we do not need this.
 An HNN-extension $\langle H,t \mid  t^{-1} a t = \varphi(a) \, (a \in A) \rangle$ can be obtained as
 the fundamental group of a graph of groups $(G,Y)$, where $Y$ consists of a single vertex $v$, a loop $e$ 
 and its inverse edge $e^{-1}$, and $G_v = H$, $G_e = A$. Similarly, amalgamated free products are special
 cases of fundamental groups.

The word problem for the fundamental group $\pi_1(G, Y)$ can be solved using a generalized form of Britton reduction \cite{DieWei13}.
This consists of applying the rewriting steps $e\omega_e(g)e\inv \to \alpha_e(g)$ for $e \in E$, $g \in G_e$
as long as possible. The proof of the following theorem is completely analogous to the proof of 
Theorem~\ref{thm-general-WP-HNN}.

\begin{theorem} \label{thm-general-WP-mHNN}
Let $(G,Y)$ be a graph of groups such that $Y = (V,E)$ is finite and 
every edge group $G_e$, $e \in E$, is cyclic
and the image $\alpha_e(G_e)$ is undistorted in $G_{\alpha(e)}$.\footnote{Then also $\omega_e(G_e)$ is undistorted in $G_{\omega(e)}$.}
Then the semi-compressed word problem for the fundamental group $\pi_1(G, Y)$ is polynomial time Turing-reducible
to the compressed power problems for the vertex groups $G_v$, $v \in V(Y)$.
\end{theorem}

\begin{corollary} \label{coro-hyp2}
Let $G$ be a fundamental group of a graph of groups such that all vertex groups are hyperbolic and
all edge groups are cyclic.  Then the word problem for $G$ can be solved in polynomial time.
\end{corollary}

\section{Future work}

There is no hope to generalize Corollary~\ref{coro-hyp} to the case of finitely generated
associated subgroups (there exists a finitely generated subgroup $A$ of a hyperbolic group $G$ such that the membership problem for 
$A$ is undecidable \cite{Rip82b}).  On the other hand, it is known that the membership problem for quasiconvex subgroups of 
hyperbolic groups is decidable. What is the complexity of the word problem for an HNN-extension of a hyperbolic group $H$ with 
finitely generated quasiconvex associated subgroups? Even for the case where $H$ is free (where all subgroups are quasiconvex)
the existence of a polynomial time algorithm is not clear. 

The best known complexity bound for the word problem of a hyperbolic group is $\LogCFL$, which is contained in the 
circuit complexity class $\NC^2$. This leads to the question whether the complexity bound in  Corollary~\ref{coro-hyp}
can be improved to $\NC$. Also the complexity of the the compressed word problem for an HNN-extension of a hyperbolic 
group $H$ with cyclic associated subgroups is open (even in the case where the  base group $H$ is free).
Recall that the compressed word problem for a hyperbolic group can be solved in polynomial time \cite{HoltLS19}. 

\paragraph{\bf Acknowledgments.} This work is supported by the DFG project LO748/12-1.


\begin{thebibliography}{10}
\providecommand{\url}[1]{\texttt{#1}}
\providecommand{\urlprefix}{URL }
\providecommand{\doi}[1]{https://doi.org/#1}

\bibitem{Artin25}
Artin, E.: Theorie der {Z}{\"o}pfe. Abhandlungen aus dem Mathematischen Seminar
  der Universit{\"a}t Hamburg  \textbf{4}(1),  47--72 (1925)

\bibitem{AvMa84a}
Avenhaus, J., Madlener, K.: The {Nielsen} reduction and {P}-complete problems
  in free groups. Theoretical Computer Science  \textbf{32}(1-2),  61--76
  (1984)

\bibitem{BiStre78}
Bieri, R., Strebel, R.: Almost finitely presented soluble groups. Commentarii
  Mathematici Helvetici  \textbf{53},  258--278 (1978)

\bibitem{bjofra05}
Bj{\"o}rner, A., Brenti, F.: Combinatorics of {Coxeter} {Groups}, Graduate
  Texts in Mathematics, vol.~231. Springer, New York (2005)

\bibitem{Boo59}
Boone, W.W.: The word problem. Annals of Mathematics. Second Series
  \textbf{70},  207--265 (1959)

\bibitem{Britt63}
Britton, J.L.: {The word problem}. Annals of Mathematics  \textbf{77}(1),
  16--32 (1963)

\bibitem{Cha07}
Charney, R.: An introduction to right-angled {Artin} groups. Geometriae
  Dedicata  \textbf{125},  141--158 (2007)

\bibitem{dehn11}
Dehn, M.: \"{U}ber unendliche diskontinuierliche {Gruppen}. Mathematische
  Annalen  \textbf{71},  116--144 (1911)

\bibitem{Dehn12}
Dehn, M.: {Transformation der Kurven auf zweiseitigen Fl\"achen.} Mathematische
  Annalen  \textbf{72},  413--421 (1912)

\bibitem{DiekertK16}
Diekert, V., Kausch, J.: Logspace computations in graph products. Journal of Symbolic
  Computation  \textbf{75},  94--109 (2016)

\bibitem{DieWei13}
Diekert, V., Wei{\ss}, A.: Context-free groups and {Bass-Serre} theory. CoRR
  \textbf{abs/1307.8297} (2013), \url{https://arxiv.org/abs/1307.8297}

\bibitem{EpsCHLPT92}
Epstein, D.B.A., Cannon, J.W., Holt, D.F., Levy, S.V.F., Paterson, M.S.,
  Thurston, W.P.: {Word Processing in Groups}. Jones and Bartlett (1992)

\bibitem{EpsteinH06}
Epstein, D.B.A., Holt, D.F.: The linearity of the conjugacy problem in
  word-hyperbolic groups. { International Journal of Algebra and Computation}
  \textbf{16}(2),  287--306 (2006)

\bibitem{ghys1990groupes}
Ghys, {\'E}., de~La~Harpe, P.: Sur les groupes hyperboliques d'apr{\`e}s
  Mikhael Gromov. Progress in mathematics, Birkh{\"a}user (1990)

\bibitem{Gro87}
Gromov, M.: Hyperbolic groups. In: Gersten, S.M. (ed.) Essays in Group Theory.
  pp. 75--263. No.~8 in MSRI Publ., Springer (1987)

\bibitem{Hag00}
Hagenah, C.: Gleichungen mit regul{\"a}ren Randbedingungen {\"u}ber freien
  Gruppen. Ph.D. thesis, University of {Stuttgart} (2000)

\bibitem{HauLo11}
Haubold, N., Lohrey, M.: Compressed word problems in {HNN}-extensions and
  amalgamated products. Theory of Computing Systems  \textbf{49}(2),  283--305
  (2011)

\bibitem{Hol00}
Holt, D.: Word-hyperbolic groups have real-time word problem. International
  Journal of Algebra and Computation  \textbf{10},  221--228 (2000)

\bibitem{HoltLS19}
Holt, D.F., Lohrey, M., Schleimer, S.: Compressed decision problems in
  hyperbolic groups. In: 36th International Symposium on Theoretical Aspects of
  Computer Science, {STACS} 2019, March 13-16, 2019, Berlin, Germany. LIPIcs,
  vol.~126, pp. 37:1--37:16. Schloss Dagstuhl - Leibniz-Zentrum f{\"{u}}r
  Informatik (2019), \url{http://www.dagstuhl.de/dagpub/978-3-95977-100-9}

\bibitem{LiZa77}
Lipton, R.J., Zalcstein, Y.: Word problems solvable in logspace. Journal of the
  ACM  \textbf{24}(3),  522--526 (1977)

\bibitem{Lo05ijfcs}
Lohrey, M.: Decidability and complexity in automatic monoids. International
  Journal of Foundations of Computer Science  \textbf{16}(4),  707--722 (2005)

\bibitem{Loh14}
Lohrey, M.: The Compressed Word Problem for Groups. SpringerBriefs in
  Mathematics, Springer (2014)

\bibitem{mag32}
Magnus, W.: Das {I}dentit\"atsproblem f\"ur {G}ruppen mit einer definierenden
  {R}elation. Mathematische Annalen  \textbf{106}(1),  295--307 (1932)

\bibitem{MaWe21}
Mattes, C., Wei{\ss}, A.:
Parallel algorithms for power circuits and the word problem of the Baumslag group.
CoRR \textbf{abs/2102.09921} (2021), \url{https://arxiv.org/abs/2102.09921}

\bibitem{Min04}
Minasyan, A.: On products of quasiconvex subgroups in hyperbolic groups.
  International Journal of Algebra and Computation  \textbf{14}(2),  173--195
  (2004)

\bibitem{MyNi14}
Myasnikov, A., Nikolaev, A.: Verbal subgroups of hyperbolic groups have
  infinite width. Journal of the London Mathematical Society  \textbf{90}(2),
  573--591 (2014)

\bibitem{MyNiUs14}
Myasnikov, A., Nikolaev, A., Ushakov, A.: Knapsack problems in groups.
  Mathematics of Computation  \textbf{84},  987--1016 (2015)

\bibitem{MyUsWo11}
Myasnikov, A.,  Ushakov, A., Won, D.W.:
\newblock The word problem in the {Baumslag} group with a non-elementary {Dehn}
  function is polynomial time decidable.
Journal of Algebra \textbf{345}(1), 324--342 (2011)

\bibitem{Nov58}
Novikov, P.S.: On the algorithmic unsolvability of the word problem in group
  theory. American Mathematical Society, Translations, II. Series  \textbf{9},
  1--122 (1958)

\bibitem{Rab60}
Rabin, M.O.: {Computable algebra, general theory and theory of computable
  fields.} Transactions of the American Mathematical Society  \textbf{95},
  341--360 (1960)

\bibitem{Rip82b}
Rips, E.: Subgroups of small cancellation groups. Bulletin of the London
  Mathematical Society  \textbf{14},  45--47 (1982)

\bibitem{Serre03}
Serre, J.P.: Trees. Springer (2003)

\bibitem{Sim79}
Simon, H.U.: Word problems for groups and contextfree recognition. In:
  Proceedings of Fundamentals of Computation Theory, FCT 1979. pp. 417--422.
  Akademie-Verlag (1979)

\bibitem{Sti95}
Stillwell, J.: Classical Topology and Combinatorial Group Theory (2nd edition).
  Springer (1995)

\bibitem{Waa90}
Waack, S.: The parallel complexity of some constructions in combinatorial group
  theory. Journal of Information Processing and Cybernetics, EIK  \textbf{26},
  265--281 (1990)

\bibitem{Wehr80}
Wehrfritz, B.A.F.: {On finitely generated soluble linear groups.} Mathematische
  Zeitschrift  \textbf{170},  155--167 (1980)

\bibitem{Weiss15b}
Wei{\ss}, A.: On the complexity of conjugacy in amalgamated products and {HNN}
  extensions. Ph.D. thesis, University of Stuttgart (2015)

\bibitem{Weiss16}
Wei{\ss}, A.: A logspace solution to the word and conjugacy problem of
  generalized {Baumslag-Solitar} groups. In: Algebra and Computer Science.
  Contemporary Mathematics, vol.~677. American Mathematical Society (2016)

\end{thebibliography}

\end{document}